\documentclass[12pt]{amsart}
\usepackage{amsfonts,graphicx,amssymb,amscd,amsmath,enumerate,verbatim,calc}
\def\NZQ{\mathbb}               

\def\ZZ{{\NZQ Z}}
\def\RR{{\NZQ R}}
\def\CC{{\NZQ C}}

\def\frk{\mathfrak}               

\def\Phi{{\frk N}}
\def\a{\alpha}

\def\opn#1#2{\def#1{\operatorname{#2}}} 
\opn\chara{char} 
\opn\length{\ell} 
\opn\pd{pd} 
\opn\rk{rk}
\opn\projdim{proj\,dim} 
\opn\injdim{inj\,dim} 
\opn\rank{rank}
\opn\depth{depth} 
\opn\grade{grade} 
\opn\height{height}
\opn\embdim{emb\,dim} 
\opn\codim{codim}

\opn\Tr{Tr} 
\opn\bigrank{big\,rank}
\opn\superheight{superheight}
\opn\lcm{lcm}
\opn\trdeg{tr\,deg}
\opn\reg{reg} 
\opn\lreg{lreg} 
\opn\ini{in} 
\opn\lpd{lpd}
\opn\size{size}
\opn\mult{mult}
\opn\dist{dist}
\opn\cone{cone}
\opn\lex{lex}
\opn\rev{rev}
\opn\div{div} \opn\Div{Div} \opn\cl{cl} \opn\Cl{Cl}
\opn\Syz{Syz} \opn\Im{Im} \opn\Ker{Ker} \opn\Coker{Coker}
\opn\Re{Re}
\opn\Am{Am} \opn\Hom{Hom} \opn\Tor{Tor} \opn\Ext{Ext}
\opn\End{End} \opn\Aut{Aut} \opn\id{id} \opn\ini{in}

\opn\nat{nat}
\opn\pff{pf}
\opn\Pf{Pf} \opn\GL{GL} \opn\SL{SL} \opn\mod{mod} \opn\ord{ord}
\opn\Gin{Gin}
\opn\Hilb{Hilb}\opn\adeg{adeg}\opn\std{std}\opn\ip{infpt}
\opn\Pol{Pol}
\opn\sat{sat}
\opn\Var{Var}
\opn\Gen{Gen}

\opn\aff{aff} \opn\con{conv} \opn\relint{relint} \opn\st{st}
\opn\lk{lk} \opn\cn{cn} \opn\core{core} \opn\vol{vol}
\opn\link{link} \opn\star{star}
\opn\gr{gr}

\def\Pc{{\mathcal P}}

\def\pot#1#2{#1[\kern-0.28ex[#2]\kern-0.28ex]}

\opn\dirlim{\underrightarrow{\lim}}
\opn\inivlim{\underleftarrow{\lim}}
\let\to=\rightarrow

\def\Implies{\ifmmode\Longrightarrow \else
        \unskip${}\Longrightarrow{}$\ignorespaces\fi}
\def\implies{\ifmmode\Rightarrow \else
        \unskip${}\Rightarrow{}$\ignorespaces\fi}
\def\iff{\ifmmode\Longleftrightarrow \else
        \unskip${}\Longleftrightarrow{}$\ignorespaces\fi}

\let\:=\colon
\newtheorem{Theorem}{Theorem}[section]

\newtheorem{Proposition}[Theorem]{Proposition}
\newtheorem{Remark}[Theorem]{Remark}

\newtheorem{Definition}[Theorem]{Definition}

\newtheorem{Conjecture}[Theorem]{Conjecture}
\newtheorem{Question}[Theorem]{Question}
\let\epsilon\varepsilon
\let\phi=\varphi
\let\kappa=\varkappa
\def\qed{\ifhmode\textqed\fi
      \ifmmode\ifinner\quad\qedsymbol\else\dispqed\fi\fi}
\def\textqed{\unskip\nobreak\penalty50
       \hskip2em\hbox{}\nobreak\hfil\qedsymbol
       \parfillskip=0pt \finalhyphendemerits=0}
\def\dispqed{\rlap{\qquad\qedsymbol}}

\opn\dis{dis}
\opn\height{height}
\opn\dist{dist}
\def\pnt{{\raise0.5mm\hbox{\large\bf.}}}

\opn\Lex{Lex}

\begin{document}
\title{
Roots of Ehrhart polynomials and \\
symmetric $\delta$-vectors 
}
\author{Akihiro Higashitani}
\thanks{
{\bf 2000 Mathematics Subject Classification:}
Primary 52B20; Secondary 52B12. \\
\;\;\;\; {\bf Keywords:}
Ehrhart polynomial, $\delta$-vector,
Gorenstein Fano polytope. \\
\;\;\;\; 
The author is supported by JSPS Research Fellowship for Young Scientists. 
}
\address{Akihiro Higashitani,
Department of Pure and Applied Mathematics,
Graduate School of Information Science and Technology,
Osaka University,
Toyonaka, Osaka 560-0043, Japan}
\email{a-higashitani@cr.math.sci.osaka-u.ac.jp}

\begin{abstract}
The conjecture on roots of Ehrhart polynomials, 
stated by Matsui et al. \cite[Conjecture 4.10]{MHNOH}, 
says that all roots $\alpha$ of the Ehrhart polynomial of 
a Gorenstein Fano polytope of dimension $d$ satisfy 
$-\frac{d}{2} \leq \Re(\alpha) \leq \frac{d}{2} -1$. 
In this paper, we observe the behaviors of roots of SSNN polynomials 
which are a wider class of the polynomials 
containing all the Ehrhart polynomials of Gorenstein Fano polytopes. 
As a result, we verify that this conjecture is true 
when the roots are real numbers or when $d \leq 5$. 
\end{abstract}

\maketitle

\section*{Introduction}

Let $\Pc \subset \RR^N$ be an integral convex polytope of dimension $d$. 
Given a positive integer $n$, we write $i(\Pc, n) = \sharp(n \Pc \cap \ZZ^N),$ 
where $n \Pc = \{ n \alpha : \alpha \in \Pc \}$. 
In \cite{Ehrhart}, Ehrhart succeeded in proving 
that the numerical function $i(\Pc,n)$ is 
a polynomial in $n$ of degree $d$ with $i(\Pc, 0) = 1$. 
We call $i(\Pc, n)$ the {\em Ehrhart polynomial} of $\Pc$. 

We define the sequence $\delta_0, \delta_1, \ldots, \delta_d$ of integers 
by the formula 
\begin{eqnarray*}
(1 - \lambda)^{d + 1} \sum_{n=0}^{\infty} i(\Pc,n) \lambda^n 
= \sum_{j=0}^d \delta_j \lambda^j. 
\end{eqnarray*}
We call the sequence $\delta(\Pc) = (\delta_0, \delta_1, \ldots, \delta_d)$ 
the {\em $\delta$-vector} of $\Pc$. 
Thus, $\delta_0 = 1$ and $\delta_1 = \sharp(\Pc \cap \ZZ^N) - (d + 1)$. 
It is also known that $\delta_d = \sharp((\Pc \setminus \partial \Pc) \cap \ZZ^N)$, 
where $\partial \Pc$ is the boundary of $\Pc$. 
In particular, $\delta_1 \geq \delta_d$. 
Each $\delta_j$ is nonnegative (\cite{StanleyDRCP}). 
Note that the Ehrhart polynomial can be expressed by using 
the $\delta$-vectors together with the binomial coefficients as follows: 
\begin{eqnarray*}\label{binomform}
i(\Pc,n)=\sum_{j=0}^d\delta_j\binom{n+d-j}{d}. 
\end{eqnarray*}
We refer the reader to \cite[Chapter 3]{BeckRobins} or \cite[Part II]{HibiRedBook} 
for the introduction to the theory of Ehrhart polynomials.

\smallskip 

On roots of Ehrhart polynomials, 
the following conjecture is remarkable: 
\begin{Conjecture}[{\cite[Conjecture 1.4]{BDDPS}}]\label{old}
All roots $\alpha$ of the Ehrhart polynomials of integral convex polytopes 
of dimension $d$ satisfy 
\begin{eqnarray}\label{ineq1}
-d \leq \Re(\alpha) \leq d-1, 
\end{eqnarray}
where $\Re(\alpha)$ denotes the real part of $\alpha \in \CC$. 
\end{Conjecture}

It is proved in \cite{BDDPS} and \cite{BD} that 
this conjecture is true when the roots are real numbers or when $d \leq 5$. 
However, this does not hold in general, that is, there exists a counterexample for this conjecture (\cite{counter, OS}). 
Braun \cite{Bra} gives the best possible norm bound of 
roots of Ehrhart polynomials with $O(d^2)$. 

\smallskip

A {\em Fano polytope} is an integral convex polytope $\Pc \subset \RR^d$ of dimension $d$ 
such that the origin of $\RR^d$ is a unique integer point belonging to $\Pc \setminus \partial \Pc$. 
A Fano polytope is called {\em Gorenstein} if its dual polytope is integral. 
Here the dual polytope $\Pc^\vee$ of a Fano polytope $\Pc$ is 
the convex polytope which consists of those $x \in \RR^d$ 
such that $\langle x, y \rangle \leq 1$ for all $y \in \Pc$, 
where $\langle x, y \rangle$ is the usual inner product of $\RR^d$. 
Note that Gorenstein Fano polytope is also said to be a {\em reflexive polytope}. 

Gorenstein Fano polytope is one of the most interesting and important objects 
from viewpoints of both algebraic geometry on toric Fano varieties 
and combinatorics on Fano polytopes. 
Moreover, Gorenstein Fano polytope has a sigfinicant property 
on its Ehrhart polynomial and $\delta$-vector. 
More precisely , for a Fano polytope $\Pc \subset \RR^d$ 
with its $\delta$-vector $(\delta_0,\delta_1,\ldots,\delta_d)$, 
it follows from \cite{Batyrev} and \cite{Hibidual} that 
the following conditions are equivalent:
\begin{itemize}
\item $\Pc$ is Gorenstein; 
\item $\delta_i=\delta_{d-i}$ for $0 \leq i \leq \lfloor \frac{d}{2} \rfloor$; 
\item $i(\Pc, n) = ( - 1 )^d i(\Pc, - n - 1)$. 
\end{itemize}

When $\Pc \subset \RR^d$ is a Gorenstein Fano polytope, 
since $i(\Pc, n) = ( - 1 )^d i(\Pc, - n - 1)$, 
the roots of $i(\Pc, n)$ are distributed symmetrically in the complex plane 
with respect to the line $\Re(z) = - \frac{1}{2}$. 
Thus, in particular, if $d$ is odd, then $- \frac{1}{2}$ is a root of $i(\Pc, n)$. 
It is known \cite[Proposition 1.8]{BHW} that if all roots $\a \in \CC$ of 
$i(\Pc, n)$ of an integral convex polytope $\Pc$ of dimension $d$ 
satisfy $\Re(\a) = - \frac{1}{2}$, then $\Pc$ is unimodularly equivalent 
to a Gorenstein Fano polytope whose volume is at most $2^d$. 
In \cite{HHO}, Gorenstein Fano polytopes whose Ehrhart polynomials have 
a reasonable root distribution are studied. 
In \cite{HK10}, the roots of the Ehrhart polynomials of smooth Fano polytopes 
with small dimensions are completely determined.

It seems to be meaningful to investigate root distributions of Ehrhart polynomials of 
Gorenstein Fano polytopes. 
In \cite{MHNOH}, the following conjecture is also proposed: 

\begin{Conjecture}[{\cite[Conjecture 4.10]{MHNOH}}]\label{new}
All roots $\alpha$ of the Ehrhart polynomials of Gorenstein Fano polytopes 
of dimension $d$ satisfy 
\begin{eqnarray}\label{ineq2}
-\frac{d}{2} \leq \Re(\alpha) \leq \frac{d}{2}-1. 
\end{eqnarray}\end{Conjecture}

This conjecture says that if we restrict the objects 
to Gorenstein Fano polytopes in Conjecture \ref{old}, 
then the range \eqref{ineq1} of real parts becomes half. 
We note that there also exists a certain counterexample 
of dimension 34. (See \cite{OS}.)

\smallskip

On many results of the studies on roots of Ehrhart polynomials, 
Stanley's nonnegativity of $\delta$-vectors \cite{StanleyDRCP} 
plays a crucial role (\cite{Bra, BD}). 
Derived from the definition \cite[Definition 1.2]{BD}, 
we define the following polynomial. 
\begin{Definition}{\em 
Given a sequence of nonnegative real numbers 
$(\delta_0,\delta_1,\ldots,\delta_d) \in \RR^{d+1}_{\geq 0}$ such that these numbers are symmetric, i.e., 
$\delta_i=\delta_{d-i}$ for $0 \leq i \leq \lfloor \frac{d}{2} \rfloor$, 
we define the polynomial 
$$f(n)=\sum_{i=0}^d \delta_i\binom{n+d-i}{d}$$ in $n$ of degree $d$. 
We call $f(n)$ a {\em symmetric Stanley's nonnegative} or 
{\em SSNN} polynomial of degree $d$. 
}\end{Definition}
In \cite[Remark 2.2]{Pfe}, this family of polynomials is mentioned, 
although it is not pursued deeply there.

As developed in \cite{BD}, studying roots of a wider class of Ehrhart polynomials is 
sometimes convenient and crucial. 
Also in this paper, we study roots of SSNN polynomials and consider the following question. 

\begin{Question}\label{q}
Do all roots $\alpha$ of an SSNN polynomial of degree $d$ satisfy 
$$- \frac{d}{2} 
\leq \Re(\alpha) \leq \frac{d}{2} -1 \; ?$$ 
\end{Question}

This is true 
when the roots are real numbers or when $d \leq 5$. In fact, 
\begin{Theorem}\label{main}
Let $f(n)$ be an SSNN polynomial of degree $d$ 
and $\alpha \in \CC$ an arbitrary root of $f(n)$. \\
{\em (a)} If $\alpha \in \RR$, then $\alpha$ satisfies 
$- \frac{d}{2} \leq \alpha \leq \frac{d}{2} -1$, more strictly, 
$-\left\lfloor \frac{d}{2} \right\rfloor 
\leq \a \leq \left\lfloor \frac{d}{2} \right\rfloor -1.$ \\
{\em (b)} If $d \leq 5$, then $\alpha$ satisfies 
$-\frac{d}{2} \leq \Re(\a) \leq \frac{d}{2} -1$, 
more strictly, 
$-\left\lfloor \frac{d}{2} \right\rfloor 
\leq \Re(\alpha) \leq \left\lfloor \frac{d}{2} \right\rfloor -1.$ 
\end{Theorem}

We prove Theorem \ref{main} in Section 1. 
Moreover, in Section 2, we discuss some differences of root distributions 
between SSNN polynomials and Ehrhart polynomials of Gorenstein Fano polytopes 
of small dimensions (degrees).


\section{A proof of Theorem \ref{main}}

This section is devoted to proving Theorem \ref{main}. 

\bigskip

Let $f(n)=\sum_{i=0}^d\delta_i\binom{n+d-i}{d}$ be an SSNN polynomial of degree $d$. 
First of all, we verify that $f(n)$ satisfies 
\begin{eqnarray}\label{funceq}
f(n)=(-1)^df(-n-1).\end{eqnarray} 
Let $$N_i(n)=\prod_{j=0}^{d-1}(n+d-i-j)+\prod_{j=0}^{d-1}(n+i-j)$$ 
for $0 \leq i \leq \lfloor \frac{d}{2} \rfloor -1$ and 
\begin{eqnarray*}
N_{\lfloor \frac{d}{2} \rfloor}(n)=
\begin{cases}
\prod_{j=0}^{d-1}(n+\frac{d}{2}-j), \quad\quad\quad\quad\quad\quad\quad\quad\quad\quad\quad
&\text{if } d \text{ is even}, \\
\prod_{j=0}^{d-1}(n+\frac{d+1}{2}-j)+\prod_{j=0}^{d-1}(n+\frac{d-1}{2}-j), 
&\text{if } d \text{ is odd}. 
\end{cases}
\end{eqnarray*}
It then follows that $f(n)=\sum_{i=0}^{\lfloor \frac{d}{2} \rfloor}\frac{\delta_iN_i(n)}{d!}.$ 
Since one has 
\begin{eqnarray*}
(-1)^dN_i(-n-1)&=&(-1)^d\prod_{j=0}^{d-1}(-n-1+d-i-j)+(-1)^d\prod_{j=0}^{d-1}(-n-1+i-j) \\
&=&\prod_{j=0}^{d-1}(n+1-d+i+j)+\prod_{j=0}^{d-1}(n+1-i+j) \\
&=&\prod_{j=0}^{d-1}(n+i-j)+\prod_{j=0}^{d-1}(n+d-i-j)=N_i(n) 
\end{eqnarray*}
for $0 \leq i \leq \lfloor \frac{d}{2} \rfloor -1$ and 
$(-1)^dN_{\lfloor \frac{d}{2} \rfloor}(-n-1)=N_{\lfloor \frac{d}{2} \rfloor}(n)$, 
we obtain the desired identity $f(n)=(-1)^df(-n-1)$. 

We give a proof of Theorem \ref{main} (a) by using the above notations. 
\begin{proof}[Proof of Theorem \ref{main} (a)]
Let $$g(n)=d!f\left(n-\frac{1}{2} \right)=
\sum_{i=0}^{\lfloor \frac{d}{2} \rfloor}\delta_iN_i\left(n-\frac{1}{2}\right).$$ 
Then, it suffices to prove that all real roots of $g(n)$ are contained in 
the closed interval 
$[-\lfloor \frac{d}{2} \rfloor +\frac{1}{2}, \lfloor \frac{d}{2} \rfloor -\frac{1}{2}]$. 
It follows from \eqref{funceq} that 
\begin{eqnarray}\label{funceq2}
g(n)=(-1)^dg(-n). 
\end{eqnarray}

For $N_i\left(n-\frac{1}{2}\right)$, when $0 \leq i \leq \lfloor \frac{d}{2}\rfloor-1 $, 
we have 
\begin{eqnarray*}
N_i\left(n-\frac{1}{2}\right)&=& 
\prod_{j=0}^{d-1}\left(n+d-\frac{1}{2}-i-j\right)
+\prod_{j=0}^{d-1}\left(n-\frac{1}{2}+i-j\right) \\
&=&\prod_{l=0}^{2i-1}\left(n-\frac{1}{2}+i-l \right)M_i(n), 
\end{eqnarray*}
where 
$$M_i(n)=\prod_{j=0}^{d-2i-1}\left( n+\frac{1}{2}+i+j \right)+
\prod_{j=0}^{d-2i-1}\left( n-\frac{1}{2}-i-j \right), $$
and 
\begin{eqnarray*}
N_{\lfloor \frac{d}{2} \rfloor}\left(n-\frac{1}{2}\right)=
\begin{cases}
\prod_{j=0}^{d-1}\left( n+\frac{d-1}{2}-j \right) \;\;\;\;\;\;\;\;&\text{if $d$ is even}, \\
2n \prod_{l=0}^{d-2}\left( n+\frac{d-2}{2}-l \right) &\text{if $d$ is odd}. 
\end{cases}
\end{eqnarray*}
Let $\alpha$ be a real number with 
$\alpha > \lfloor \frac{d}{2} \rfloor -\frac{1}{2}$. 
Since the coefficients of $n^j$ in $M_i(n)$, where $0 \leq j \leq d-2i-1$, are all nonnegative, 
we have $M_i(\alpha) > 0$. In addition, 
one has $\prod_{l=0}^{2i-1}\left(\alpha- \left( \frac{1}{2}-i+l \right) \right) > 0$ 
because $0 \leq l \leq 2i-1$ and $0 \leq i \leq \lfloor \frac{d}{2} \rfloor$. 
Furthermore, $N_{\lfloor \frac{d}{2} \rfloor}(\alpha-\frac{1}{2})=1$ if $d$ is even and 
$N_{\lfloor \frac{d}{2} \rfloor}(\alpha-\frac{1}{2})=2\alpha$ if $d$ is odd. 
Hence, $\alpha$ cannot be a root of $g(n)$ from the nonnegativity of 
$\delta_0,\delta_1,\ldots,\delta_{\lfloor \frac{d}{2} \rfloor}$. 
Moreover, by virtue of (\ref{funceq2}), for a real number 
$\beta$ with $\beta < -\left\lfloor \frac{d}{2} \right\rfloor + \frac{1}{2}$, 
$\beta$ cannot be a root of $g(n)$, as desired. 
\end{proof}

In the rest of this section, we prove Theorem \ref{main} (b).

\subsection{The case where $d=2$ and $3$}

\begin{itemize}
\item An SSNN polynomial of degree 2 has two roots. 
If both of them are real numbers, then the assertion holds from Theorem \ref{main} (a). 
If both of them are non-real numbers, then it follows from (\ref{funceq}) that 
each of their real parts must be $-\frac{1}{2}$.

\item An SSNN polynomial of degree 3 has three roots and one is $-\frac{1}{2}$. 
On the other two roots, the same discussion as the case where $d=2$ can be done. 
\end{itemize}

\subsection{The case where $d=4$}

Let $f(n)=\frac{a}{4!}N_0(n)+\frac{b}{4!}N_1(n)+\frac{c}{4!}N_2(n)$, 
where $a,b,c \in \RR_{\geq 0}$. 
Then $f(n)$ has four roots and the possible cases are as follows: 
\begin{itemize}
\item[(i)] those four roots are all real numbers; 
\item[(ii)] two of them are real numbers and the others are non-real numbers; 
\item[(iii)] those four roots are all non-real numbers. 
\end{itemize}
We do not have to discuss the cases (i) and (ii) by virtue of Theorem \ref{main} (a). 
Thus, we consider the case (iii), i.e., we assume that $f(n)$ has four non-real roots. 
Moreover, we may also assume that $a \not=0$ 
since both $0$ and $-1$ are their roots when $a=0$. 
In addition, we may set $a=1$ since the roots of $f(n)$ 
exactly coincide with those of $\frac{f(n)}{a}$. 

We define 
\begin{eqnarray*}
g(n)=4!f\left( n-\frac{1}{2} \right) 
= (2+2b+c)n^4 + \left(43+7b-\frac{5}{2}c \right)n^2 + \frac{105}{8}-\frac{15}{8}b+\frac{9}{16}c. 
\end{eqnarray*}
Our work is to show that if the roots $\alpha$ of $g(n)$ are all non-real numbers, 
then $\alpha$ satisfies $-\frac{3}{2} \leq \Re(\alpha) \leq \frac{3}{2}$. Let 
$$G(X)=(2+2b+c)X^2 + \left(43+7b-\frac{5}{2}c \right)X + \frac{105}{8}-\frac{15}{8}b+\frac{9}{16}c.$$ 
We consider the roots of $G(X)$. 
Let $\alpha$ and $\beta$ be the roots of $G(X)$ and $D(G(X))$ the discriminant. 
By our assumption, we may set $D(G(X)) < 0$. 
In fact, when $D(G(X)) \geq 0$, i.e., both $\alpha$ and $\beta$ are real numbers, 
then the roots of $g(n)$ are $\pm \sqrt{\alpha}, \pm\sqrt{\beta}$. 
Even if $\alpha$ (resp. $\beta$) is positive or negative, 
$\pm \sqrt{\alpha}$ (resp. $\pm \sqrt{\beta}$) are 
either real numbers or pure imaginary numbers. 

Let, say, $\alpha=re^{\theta \sqrt{-1}}$ with $r > 0$ and $0 < \theta < \pi$. 
Then $\beta=\bar{\alpha}=re^{-\theta \sqrt{-1}}$. 
Thus the roots of $g(n)$ are $\sqrt{r}e^{\pm \frac{\theta}{2} \sqrt{-1}}$ and 
$\sqrt{r}e^{\pm (\pi-\frac{\theta}{2}) \sqrt{-1}}.$ 
Hence, it is enough to show that 
$$ 0 < \Re(\sqrt{r}e^{\frac{\theta}{2} \sqrt{-1}}) = 
\sqrt{r}\cos \frac{\theta}{2} = \sqrt{r}\sqrt{\frac{1+\cos \theta}{2}} = 
\sqrt{\frac{r+r \cos \theta}{2}} \leq \frac{3}{2}.$$ 
Since $G(X)=(2+2b+c)(X-\alpha)(X-\beta)$, we have $$r=\frac{1}{4}\sqrt{\frac{210-30b+9c}{2+2b+c}}.$$ 


By the way, one has 
\begin{eqnarray*}
D(G(X)) &=& \left(43+7b-\frac{5}{2}c\right)^2- 
4(2+2b+c)\left( \frac{105}{8}-\frac{15}{8}b+\frac{9}{16}c \right) \\
&=& 4(16b^2-8(c-16)b+c^2-68c+436). \end{eqnarray*} 
Let $h(b)=\frac{D(G(X))}{4}$. Then $h(b) < 0$ by our assumption. 
The range of $b$ with $h(b)<0$ is 
$$\frac{4(c-16)-\sqrt{D(h(b))}}{16} < b < \frac{4(c-16)+\sqrt{D(h(b))}}{16},$$ 
where $4 \cdot D(h(b))=4 \cdot 24^2(c-5)$. (In particular, it must be $c \geq 5$.) 
Moreover, since $b \geq 0$, it must be $\frac{c-16+6 \sqrt{c-5}}{4} > 0$. 
Thus, $c > 34 - 12 \sqrt{5} (>5)$. 
Hence, the condition $D(G(X))<0$ is equivalent to the following conditions: 
\begin{eqnarray}\label{condi1}
\nonumber &&c > 34 - 12 \sqrt{5}, \\
&&\begin{cases}
0 \leq b < \frac{c-16+6\sqrt{c-5}}{4}, \;\; 
&\text{when } \;\; 34 - 12 \sqrt{5} < c \leq 34 + 12 \sqrt{5}, \\
\frac{c-16-6\sqrt{c-5}}{4} < b < \frac{c-16+6\sqrt{c-5}}{4}, 
&\text{when } \;\; c > 34 + 12 \sqrt{5}. 
\end{cases}
\end{eqnarray}

When $b$ and $c$ satisfy the first condition of (\ref{condi1}), we have 
\begin{align*}
\frac{210-30b+9c}{2+2b+c}&= \frac{-15(2+2b+c)+24c+240}{2+2b+c} =-15+24 \cdot \frac{c+10}{2+2b+c} \\
&\leq -15+24 \cdot \frac{c+10}{c+2} < -15+24 \cdot \frac{34 - 12 \sqrt{5}+10}{34 - 12 \sqrt{5}+2} \\
&= 21+4 \sqrt{5} =(2 \sqrt{5}+1)^2. 
\end{align*} 
Similarly, when $b$ and $c$ satisfy the second condition of \eqref{condi1}, we have 
\begin{align*}
\frac{210-30b+9c}{2+2b+c}&= -15+24 \cdot \frac{c+10}{2+2b+c} < -15+24 \cdot \frac{c+10}{\frac{c-16-6\sqrt{c-5}}{2}+c+2} \\
&= -15+16 \cdot\frac{c+10}{c-4-2\sqrt{c-5}}(=:H(c)) \\
&< -15+16 \cdot \frac{34 + 12 \sqrt{5}+10}{34 + 12 \sqrt{5}-4-2\sqrt{34 + 12 \sqrt{5}-5}}, \\
&\left( \text{since } \frac{dH(c)}{dc} < 0 \text{ when } c > 34 + 12 \sqrt{5},\right) \\
&= -15+16 \cdot \frac{22 + 6 \sqrt{5}}{15 + 6 \sqrt{5}-\sqrt{29 + 12 \sqrt{5}}} < (2 \sqrt{5}+1)^2. \\
\end{align*} Therefore, one has 
$$\sqrt{\frac{r+r \cos \theta}{2}} \leq \sqrt{r} 
=\frac{1}{2}\left( \frac{210-30b+9c}{2+2b+c} \right)^{\frac{1}{4}} 
< \frac{\sqrt{2 \sqrt{5}+1}}{2}<\frac{3}{2},$$ as required.


\subsection{The case where $d=5$}

This case is similar to the previous case. 
Let $f(n)=\frac{a}{5!}N_0(n)+\frac{b}{5!}N_1(n)+\frac{c}{5!}N_2(n)$, 
where $a,b,c \in \RR_{\geq 0}$. 
Then $f(n)$ has five roots and one of them is $-\frac{1}{2}$. 
For the other roots, the possible cases are the same as 
the case where $d=4$. 
We discuss only the case (iii) and assume that $a=1$. We define 
\begin{eqnarray*}
g(n)=\frac{5!f\left( n-\frac{1}{2} \right)}{n} 
= 2(1+b+c)n^4 + 5(23+7b-c) n^2 + \frac{1689}{8}-\frac{71}{8}b+\frac{9}{8}c. 
\end{eqnarray*}
Our work is to show that if the roots $\alpha$ of $g(n)$ are all non-real numbers, 
then $\alpha$ satisfies $-\frac{3}{2} \leq \Re(\alpha) \leq \frac{3}{2}$. 
Let $$G(X)=2(1+b+c)X^2 + 5(23+7b-c) X + \frac{1689}{8}-\frac{71}{8}b+\frac{9}{8}c.$$ 
We consider the roots of $G(X)$ and prove 
$\sqrt{\frac{r+r \cos \theta}{2}} \leq \frac{3}{2}$ under the assumption $D(G(X))<0$, 
where $\alpha=re^{\theta \sqrt{-1}} (r>0, 0< \theta < \pi)$ is one of the roots of $G(X)$. 
Note that 
$$r= \frac{1}{4}\sqrt{\frac{1689-71b+9c}{1+b+c}} \;\; \text{ and } \;\; 
r \cos \theta = -\frac{1}{4} \cdot \frac{115+35b-5c}{1+b+c}.$$


By the way, one has 
\begin{eqnarray*}
D(G(X)) &=& 25 (23+7b-c)^2 - 
8(1+b+c)\left( \frac{1689}{8}-\frac{71}{8}b+\frac{9}{8}c \right) \\
&=& 16(81b^2-6(3c-67)b+c^2-178c+721). \end{eqnarray*} 
Let $h(b)=\frac{D(G(X))}{16}$. Then $h(b) < 0$ by our assumption. 
The range of $b$ with $h(b)<0$ is 
$$\frac{3(3c-67)-\sqrt{D(h(b))}}{81} < b < \frac{3(3c-67)+\sqrt{D(h(b))}}{81},$$ 
where $4 \cdot D(h(b))=4 \cdot 60^2 (3c-5)$. 
(In particular, it must be $c \geq \frac{5}{3}$.) 
Moreover, since $b \geq 0$, it must be $\frac{3c-67+20\sqrt{3c-5}}{27} > 0$. 
Thus, $c > 89-60\sqrt{2} (> \frac{5}{3})$. 
Hence, the condition $D(G(X))<0$ is equivalent to the following conditions: 
\begin{eqnarray}\label{condi2}
\nonumber &&c > 89-60\sqrt{2}, \\
&&\begin{cases}
0 \leq b < \frac{3c-67+20\sqrt{3c-5}}{27}, \;\; 
&\text{when } \; 89-60\sqrt{2} < c \leq 89+60\sqrt{2}, \\
\frac{3c-67-20\sqrt{3c-5}}{27} < b < \frac{3c-67+20\sqrt{3c-5}}{27}, 
&\text{when } \; c > 89+60\sqrt{2}. 
\end{cases}
\end{eqnarray}

When $b$ and $c$ satisfy the first condition of \eqref{condi2}, we have 
\begin{align*}
&\sqrt{\frac{1689-71b+9c}{1+b+c}}+\frac{-115-35b+5c}{1+b+c} \\
&\quad\quad= \sqrt{\frac{-71(1+b+c)+80c+1760}{1+b+c}}+\frac{-35(1+b+c)+40c-80}{1+b+c} \\
&\quad\quad= \sqrt{-71+80 \cdot \frac{c+22}{1+b+c}}+
40 \cdot \frac{c-2}{1+b+c}-35 \\
&\quad\quad\leq \sqrt{-71+80 \cdot \frac{c+22}{c+1}}+
40 \cdot \frac{c-2}{c+1}-35 (=:H_1(c)) \\
&\quad\quad\leq \sqrt{-71+80 \cdot \frac{41+22}{41+1}}+40 \cdot \frac{41-2}{41+1}-35, \\
&\left(\text{since } \frac{dH_1(c)}{dc}>0 \text{ when } c<41 \text{ and } 
\frac{dH_1(c)}{dc}<0 \text{ when } c>41,\right) \\
&\quad\quad= \sqrt{-71+120}+\frac{260}{7}-35 = \frac{64}{7}. 
\end{align*} Thus, one has 
$$\sqrt{\frac{r+r \cos \theta}{2}} \leq  \sqrt{\frac{1}{2} \cdot \frac{1}{4} \cdot \frac{64}{7}}=\frac{2 \sqrt{14}}{7} < \frac{3}{2}.$$ 
On the other hand, when $b$ and $c$ satisfy the second condition of \eqref{condi2}, we have 
\begin{align*}
\frac{1689-71b+9c}{1+b+c}&= -71+80 \cdot \frac{c+22}{1+b+c} \\
&< -71+80 \cdot \frac{c+22}{\frac{3c-67-20 \sqrt{3c-5}}{27}+c+1} \\
&= -71+8 \cdot \frac{27(c+22)}{3c-4-2\sqrt{3c-5}}(=:H_2(c)) \\
&< -71+8 \cdot \frac{27(89+60\sqrt{2}+22)}{3(89+60\sqrt{2})-4-2\sqrt{3(89+60\sqrt{2})-5}} \\
&\left( \text{since } \frac{dH_2(c)}{dc} < 0 \text{ when } c > 89+60\sqrt{2}, \right) \\
&=177-112\sqrt{2}. 
\end{align*} Therefore, one has 
$$\sqrt{\frac{r+r \cos \theta}{2}} \leq \sqrt{r} 
=\frac{1}{2}\left( \frac{1689-71b+9c}{1+b+c} \right)^{\frac{1}{4}} < \frac{1}{2} \left( 177-112\sqrt{2} \right)^{\frac{1}{4}} 
< \frac{2\sqrt{14}}{7} < \frac{3}{2},$$ as required.

\begin{Remark}\label{counter}{\em 
There is an SSNN polynomial of degree 8 one of whose roots $\alpha$ 
does not satisfy $-4 \leq \Re(\alpha) \leq 3$. 
In fact, if we set 
$(\delta_0,\delta_1,\ldots,\delta_8)=(1,0,0,0,14,0,0,0,1)$ and 
$f(n)=\sum_{i=0}^8 \delta_i \binom{n+8-i}{8}$, 
then the roots of $f(n)$ are approximately 
\begin{eqnarray*}
&&-0.5 \pm 0.44480014 \sqrt{-1}, \;\; -0.5 \pm 1.78738687 \sqrt{-1}, \\
&&3.00099518 \pm 5.29723208 \sqrt{-1} \; \text{ and } \; 
-4.00099518 \pm 5.29723208 \sqrt{-1}. 
\end{eqnarray*}
On the other hand, $f(n)$ cannot be the Ehrhart polynomial of 
any Gorenstein Fano polytope of dimension 8 since $\delta_1 < \delta_8$. 
When $d=10$, however, there are some candidates 
of counterexamples of Conjecture \ref{new}. 
For example, let 
$(\delta_0,\delta_1,\ldots,\delta_{10})=(1,1,1,1,1,23,1,1,1,1,1)$ and 
$f(n)=\sum_{i=0}^{10} \delta_i \binom{n+10-i}{10}$. 
Then one of approximate roots of $f(n)$ is 
$$4.02470021 + 8.22732653 \sqrt{-1}.$$ 
On the other hand, in a recent paper \cite{OS}, 
a certain counterexample of Conjecture \ref{new} is provided. 
There exists a Gorenstein Fano polytope of dimension 34 
whose Ehrhart polynomial has a root $\alpha$ 
which violates $-17 \leq \Re(\alpha) \leq 16$. 
}\end{Remark}


\section{Some comparisons of SSNN polynomials 
with Ehrhart polynomials of Gorenstein Fano polytopes}

In this section, we discuss some differences of root distributions 
between Ehrhart polynomials of Gorenstein Fano polytopes and SSNN polynomials 
when $d \leq 4$. We determine the complete range of the roots of 
Ehrhart polynomials of Gorenstein Fano polytopes and SSNN polynomials 
when $d=2$ and 3 (Proposition \ref{dim2} and \ref{dim3}). 
Moreover, we see in Theorem \ref{real} that 
each of the real numbers in the closed interval $[-\frac{d}{2}, \frac{d}{2}-1]$ 
can be a real root of some SSNN polynomial of degree $d$. 

\bigskip

\begin{Proposition}\label{dim2}
{\em (a)} The set of the roots of the Ehrhart polynomials of 
Gorenstein Fano polytopes of dimension 2 coincides with 
$$\left\{-\frac{2}{3}, -\frac{1}{2}, -\frac{1}{3}\right\} \cup 
\left\{ -\frac{1}{2} \pm \frac{1}{2}\sqrt{\frac{6-i}{i+2}}\sqrt{-1} \in \CC : 
i=1,2,\ldots,5 \right\}.$$ 
{\em (b)} The set of the roots of SSNN polynomials of degree 2 
coincides with 
$$[-1,0] \cup 
\left\{ \alpha \in \CC : \Re(\alpha)=-\frac{1}{2}, \; 
0 < |\Im(\alpha)| \leq \frac{\sqrt{3}}{2}\right\}.$$ 
\end{Proposition}
\begin{proof}
Let $f(n)=\binom{n+2}{2} + b \binom{n+1}{2} + \binom{n}{2}$. 
Then $2f(n)=(b+2)n^2+(b+2)n+2$. Thus its roots are 
$$n=\frac{-(b+2) \pm \sqrt{(b+2)(b-6)}}{2(b+2)}
=-\frac{1}{2} \pm \frac{1}{2}\sqrt{\frac{b-6}{b+2}}.$$ 

It is well known that $(1,b,1) \in \ZZ^3$ is 
the $\delta$-vector of some Gorenstein Fano polytope of dimension 2 
if and only if $b \in \{1,2,\ldots,7\}$. 
Hence we obtain the assertion (a). 

On the other hand, when $b \in \RR_{\geq 0}$, 
the set of the roots of $f(n)$ coincides with 
$$(-1,0) \cup 
\left\{ \alpha \in \CC : \Re(\alpha)=-\frac{1}{2}, \; 
0 < |\Im(\alpha)| \leq \frac{\sqrt{3}}{2}\right\}.$$ 
In fact, the function $\frac{1}{2}\sqrt{\frac{b-6}{b+2}}$ is monotone increasing 
and $\lim_{b \to +\infty} \frac{1}{2}\sqrt{\frac{b-6}{b+2}}=\frac{1}{2}$ 
when $b \geq 6$, and $\frac{1}{2}\sqrt{\frac{6-b}{b+2}}$ is monotone decreasing 
when $0 \leq b <6$. 
Moreover, $-1$ and 0 are the roots of $\binom{n+1}{2}$. 
Therefore, the assertion (b) holds, as desired. 
\end{proof}

\begin{Proposition}\label{dim3}
{\em (a)} The set of the roots of the Ehrhart polynomials of 
Gorenstein Fano polytopes of dimension 3 coincides with 
\begin{eqnarray*}
&&\left\{-\frac{1}{2} \pm \frac{1}{2}\sqrt{\frac{i-23}{i+1}} \in \RR : 
i=23,24,\ldots,32,35 \right\} \cup \\
&&\quad\quad\quad\quad\quad\quad\quad 
\left\{ -\frac{1}{2} \pm \frac{1}{2}\sqrt{\frac{23-i}{i+1}}\sqrt{-1} \in \CC : 
i=1,2,\ldots,22 \right\}. 
\end{eqnarray*}
{\em (b)} The set of the roots of SSNN polynomials of degree 3 
coincides with 
$$[-1,0] \cup 
\left\{ \alpha \in \CC : \Re(\alpha)=-\frac{1}{2}, \; 
0 < |\Im(\alpha)| \leq \frac{\sqrt{23}}{2} \right\}.$$
\end{Proposition}
\begin{proof}
Let $f(n)=\binom{n+3}{3} + b \binom{n+2}{3} + b \binom{n+1}{3} + \binom{n}{3}$. 
Then $3!f(n)=(2n+1)((b+1)n^2+(b+1)n+6)$. Thus its roots are $n=-\frac{1}{2}$ and 
$$n=\frac{-(b+1) \pm \sqrt{(b+1)(b-23)}}{2(b+1)} 
=-\frac{1}{2} \pm \frac{1}{2}\sqrt{\frac{b-23}{b+1}}.$$ 

By the complete classification of Kreuzer and Skarke \cite{KS98}, 
we know that $(1,b,b,1) \in \ZZ^4$ is the $\delta$-vector of 
some Gorenstein Fano polytope of dimension 3 
if and only if $b \in \{1,2,\ldots,35\} \setminus \{33,34\}$. 
(See also {\tt http://tph16.tuwien.ac.at/kreuzer/CY/}.) 

The rest parts are similar to Proposition \ref{dim2}. 
\end{proof}

From the proof of Theorem \ref{main} (b) of the case where $d=4$ and 5 
together with the norm bound \cite{Bra}, we also obtain the following 
\begin{Proposition}\label{dim45}
{\em (a)} The roots of SSNN polynomials of degree 4 are contained in 
$$[-2,1] \cup \left\{ \alpha \in \CC \setminus \RR : 
\left|\alpha + \frac{1}{2}\right| \leq 14, \; 
\left|\Re(\alpha) + \frac{1}{2}\right| \leq \frac{\sqrt{2 \sqrt{5}+1}}{2} \right\}.$$ 
{\em (b)} The roots of SSNN polynomials of degree 5 are contained in 
$$[-2,1] \cup \left\{ \alpha \in \CC \setminus \RR : 
\left|\alpha + \frac{1}{2}\right| \leq \frac{45}{2}, \; 
\left|\Re(\alpha) + \frac{1}{2}\right| \leq \frac{2 \sqrt{14}}{7} \right\}.$$ 
\end{Proposition}

Finally, we prove 
\begin{Theorem}\label{real}
Each real number contained in the closed interval $[-\frac{d}{2}, \frac{d}{2}-1]$ 
can be realized as a root of some SSNN polynomial of degree $d$. 
\end{Theorem}
\begin{proof}
First, let us consider the case where $d$ is even. Let $k=\frac{d}{2}$ and 
$$f_0(n)=\binom{n+k+1}{d}+a\binom{n+k}{d}+\binom{n+k-1}{d},$$ 
where $a$ is a real numer with $a \geq \frac{2(2k+1)}{2k-1}$. 
Then $f_0(n)$ is an SSNN polynomial of degree $d$. Let 
$$g_0(n)=\frac{d!}{\prod_{j=-k+2}^{k-1}(n+j)}f_0(n)
=(a+2)n^2+(a+2)n+ak(-k+1)+2k^2.$$ 
From $a \geq \frac{2(2k+1)}{2k-1}$, we have $$D(g_0(n))=(2k-1)(a+2)((2k-1)a-2(2k+1)) \geq 0.$$ 
Thus both of the roots of $g_0(n)$ are real numbers and those are 
\begin{eqnarray*}
n &=& \frac{-(a+2) \pm \sqrt{(2k-1)(a+2)((2k-1)a-2(2k+1))}}{2(a+2)} \\
&=& -\frac{1}{2} \pm \frac{1}{2}\sqrt{\frac{(2k-1)((2k-1)a-2(2k+1))}{a+2}} \\
(&=:& -\frac{1}{2} \pm h_0(a) \;). 
\end{eqnarray*}
Now the function $h_0(a)$ on $a$ is monotone increasing and 
$$\lim_{a \to +\infty} h_0(a) = \frac{2k-1}{2}=k-\frac{1}{2}.$$ 
Hence, for $a \geq \frac{2(2k+1)}{2k-1}$, all the roots of each $f_0(n)$ 
are contained in the open interval $(-k, k-1)$. 
Moreover, $-k$ and $k-1$ are roots of $\binom{n+k}{d}$, 
which is an SSNN polynomial of degree $d$. 

Next, let us consider the case where $d$ is odd. Let $k=\frac{d-1}{2}$ and 
$$f_1(n)=\binom{n+k+2}{d}+a\binom{n+k+1}{d}+a\binom{n+k}{d}+\binom{n+k-1}{d},$$ 
where $a$ is a real number with $a \geq \frac{12k^2+12k-1}{(2k-1)^2}$. 
Then $f_1(n)$ is an SSNN polynomial of degree $d$. Let 
$$g_1(n)=\frac{d!}{(2n+1) \prod_{j=-k+2}^{k-1}(n+j)}f_2(n) 
=(a+1)n^2+(a+1)n+ak(-k+1)+3k(k+1).$$ 
From $a \geq \frac{12k^2+12k-1}{(2k-1)^2}$, we have $$D(g_1(n))=(a+1)((2k-1)^2a-(12k^2+12k-1)) \geq 0. $$ 
Thus both of the roots of $g_1(n)$ are real numbers and those are 
\begin{eqnarray*}
n &=& \frac{-(a+1) \pm \sqrt{(a+1)((2k-1)^2a-(12k^2+12k-1))}}{2(a+1)} \\
&=& -\frac{1}{2} \pm \frac{1}{2}\sqrt{\frac{(2k-1)^2a-(12k^2+12k-1)}{a+1}} \\
(&=:& -\frac{1}{2} \pm h_1(a) \;). 
\end{eqnarray*}
Now the function $h_1(a)$ on $a$ is monotone increasing and 
$$\lim_{a \to +\infty} h_1(a) =\frac{2k-1}{2}=k-\frac{1}{2}.$$ 
Hence, for $a \geq \frac{12k^2+12k-1}{(2k-1)^2}$, all the roots of each $f_1(n)$ 
are contained in the open interval $(-k, k-1)$. 
Moreover, $-k$ and $k-1$ are roots of 
$\binom{n+k+1}{d}+\binom{n+k}{d}$, which is an SSNN polynomial of degree $d$. 
\end{proof}

\end{document}